\documentclass[12pt]{article}
\usepackage{amsmath,amssymb,epsfig,setspace,color,amsthm,enumerate, tikz, subfigure}
\usepackage[left=2.3cm,top=2.3cm,right=2.3cm,bottom=2.3cm]{geometry}
\newtheorem{theorem}{Theorem}[section]
\newtheorem{lemma}[theorem]{Lemma}
\newtheorem{observation}[theorem]{Observation}
\newtheorem{question}[theorem]{Question}
\newtheorem{corollary}[theorem]{Corollary}
\newtheorem{proposition}[theorem]{Proposition}

\theoremstyle{definition}
\newtheorem{definition}[theorem]{Definition}

\title{The damage number of the Cartesian product of graphs}
\author{Melissa A. Huggan\thanks{Department of Mathematics, Vancouver Island University, Nanaimo, BC, Canada}, M.E. Messinger\thanks{Department of Mathematics and Computer Science, Mount Allison University, Sackville, NB, Canada}, Amanda Porter$^\dagger$}
\date{\today}
\begin{document}

\maketitle

\begin{abstract} We consider a variation of Cops and Robber where vertices visited by a robber are considered damaged and a single cop aims to minimize the number of distinct vertices damaged by a robber.  Motivated by the interesting relationships that often emerge between input graphs and their Cartesian product, we study the damage number of the Cartesian product of graphs.  We provide a general upper bound and consider the damage number of the product of two trees or cycles.   We also consider graphs with small damage number.
\end{abstract} 

\noindent keywords and phrases: pursuit-evasion on graphs, Cops and Robber, damage number, Cartesian product\medskip


\section{Introduction}

{\it Cops and Robber} is a pursuit-evasion game played on a graph with two players: a set of  {\it cops} and a single {\it robber}.  Initially, each cop chooses a vertex to occupy and then the robber chooses a vertex to occupy.  The game is played over a sequence of discrete time-steps, or {\it rounds}, and in each round, the cops move and then the robber moves, with the initial choice of positions considered to be round zero.   For the cops' move, each of the cops either moves to an adjacent vertex or remains at the vertex currently occupied (referred to as a {\it pass}).  The robber's move is defined similarly.  The cops win, if after some finite number of rounds, a cop occupies the same vertex as the robber: the cops have {\it captured} the robber.  If the robber can avoid capture indefinitely, the robber wins.  

Until recently, Cops and Robber has mostly been studied in the context of determining the minimum number of cops required to capture the robber (see~\cite{BN2011}).  A characterization of graphs for which one cop can capture the robber has long been known and many advances have been made towards a full characterization of graphs for which $k > 1$ cops are required to capture the robber. For graphs in which one cop can capture the robber, recent questions have included: ``how long until capture occurs?'' and ``how many distinct vertices can the robber visit?''. The former relates to the {\it capture time} of a graph and the latter to the {\it damage number} of a graph.  In this paper, we focus on the damage number of a graph.  This can be cast as a variation of the classic game of Cops and Robber, where the robber damages each distinct vertex that they visit. The goal of the robber is no longer to avoid capture, but to damage as many distinct vertices as possible (capture may or may not occur).  Naturally, the cop player aims to minimize the damage to the graph.  The robber damages a vertex by passing or moving to a neighbouring vertex. The damage number of a graph $G$, denoted ${\rm dmg}(G)$, was introduced by Cox and Sanaei~\cite{CS2019} in 2019; the motivation being scenarios where either the damage done by an intruder is costly or where there are insufficient cops to capture the robber.  In these situations, minimizing damage becomes the cops' priority since capture may be slow or not possible.  More recent results on the damage number appear in~\cite{throttling, multi-robber, SW}.

Interesting relationships are often found between input graphs and their Cartesian product, and this motivates most of our work.  In Section~\ref{sec:products}, we consider the damage number of the Cartesian product of graphs.  We provide an upper bound on dmg$(G \square H)$ for the case where at least one of $G,H$ have universal vertices before providing an upper bound for when $G,H$ are arbitrary input graphs.  In Sections~\ref{subsec:trees} and~\ref{subsec:cycles}, we determine the damage number exactly for the product of two finite trees and the product of two finite cycles, respectively. In Section~\ref{sec:small}, we completely characterize graphs with damage number $1$ and $2$. 

The remainder of this section is devoted to introducing concepts and definitions that will be used throughout the paper.  We assume all graphs are finite, non-empty, undirected, and connected.  A vertex $v$ is {\it damaged} by the robber if the robber occupies $v$ in round $i \geq 0$ and either passes or moves to a neighbouring vertex in the next round.  We note that in some Cops and Robber literature, including~\cite{CS2019}, graphs are assumed to be reflexive (i.e. every vertex contains a loop) and in such graphs, a {\it pass} equates to traversing the loop.  In other Cops and Robber literature, any subset of the cops and the robber are permitted to pass in any round and that is our assumption in this paper.  The damage number of a graph is the minimum number of distinct vertices that can be damaged by the robber.  To minimize the damage, a cop must immediately move to capture a robber on a complete graph; thus, there are some graphs for which capture is necessary to minimize damage.  In other graphs, such as trees, whether capture occurs is unimportant: it is easy to see that the damage is minimized when the cop initially occupies a center and then moves along the shortest path towards the robber.  However, once the robber occupies a leaf and the cop occupies the neighbouring stem vertex, the damage remains unchanged regardless of whether the cop moves to capture the robber or passes indefinitely. In contrast, for some graphs, capture is impossible.  In both $C_4$ and $C_5$, an optimal strategy for the cop is simply to pass at every round. For $C_4$, this strategy forces the robber to pass at every round (or be captured immediately). For $C_5$, the robber will be able to damage two vertices regardless of the cops' strategy; the cop can guard the remaining three vertices by passing unless the robber moves adjacent to the cop.  As a more interesting example, consider the cycle $C_6$, labeled according to Figure~\ref{fig:C6}.  Suppose the cop, C, initially occupies $v_1$ and the robber, R, occupies $v_5$ and consider the following cop-strategy:  the cop initially passes; if the robber moves to $v_4$, the cop moves to $v_2$ (preventing the robber from moving to $v_3$); and if the robber moves to $v_5$, the cop moves to $v_1$ (preventing the robber from moving to $v_6$).  Using this cop-strategy, which we call the \emph{oscillation strategy}, the damage is restricted to two vertices on $C_6$.  

\begin{figure}[htbp]

\[ \includegraphics[width=0.2\textwidth]{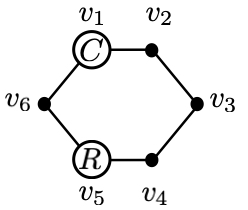} \]

\caption{An example of initial positions on $C_6$.}\label{fig:C6} 
\end{figure}

In the game of Cops and Robber, a graph is referred to as a {\it copwin} graph if one cop suffices to capture the robber.  For a copwin graph, the length of the game is $t$ if the robber is captured in the $t$-th round and a cop-strategy is optimal if its length is the minimum over all possible cop-strategies, assuming the robber is trying to avoid capture as long as possible. The invariant is denoted ${\rm capt}(G)$ and is called the {\it capture time} of $G$. Observe that ${\rm dmg}(G) \leq {\rm capt}(G)-1$~\cite{CS2019}. Note that the upper bound is meaningless for any graph $G$ that is not copwin.  

Recall that for some graphs, capture does not occur when the cop aims to minimize the damage.  In the case where the robber is eventually restricted to a single vertex, the cop must, at that point, be adjacent to all neighbours of the robber. Graph theoretically, this is described in the next definition as the robber being forced to occupy an \emph{o-dominated} vertex.  

\begin{definition}[\cite{CN2005}]
A vertex $u$ of a graph $G$ is \emph{o-dominated} if there exists a vertex $v \in V(G)$ such that $N(u) \subseteq N(v)$.  A vertex $u$ of a graph $G$ is \emph{c-dominated} if there exists a vertex $v \in V(G)$ such that $N[u] \subseteq N[v]$.  
\end{definition}
A simple example is $C_4$: the vertex occupied by the robber is $o$-dominated by the vertex occupied by the cop. 
Note that c-dominated is synonymous with the term {\it corner}, which is often used in Cops and Robber literature. For ease of reading we use the term c-dominated because the distinction between c-dominated and o-dominated vertices will come up often. 

At times we will need the definition of \emph{dominated}, where the edge between $u$ and $v$ may or may not exist. The definition is as follows.

\begin{definition}[\cite{Dahl1995}]
A vertex $u$ of a graph $G$ is \emph{dominated} if there exists a vertex $v \in V(G)$ such that $N(u) \subseteq N[v]$. In such a case, we say that $v$ dominates $u$. 
\end{definition}

If a graph $G$ contains no c-dominated vertices, then whenever the cop moves to a vertex adjacent to the robber, there is a (possibly already damaged) vertex to which the robber can move, that is not adjacent to the cop.  Consequently, the robber can avoid capture, yet damage ${\rm dmg}(G)$ vertices.  We formally state this result as it will be used later.   

\begin{observation}\label{obs:no-c} Let $G$ be a graph with no c-dominated vertices.  There exists a strategy for the robber to damage ${\rm dmg}(G)$ vertices and avoid capture.\end{observation}

To conclude this section, we state a useful lower bound.  If a graph $G$ has radius $1$, then the robber is immediately captured, no vertices are damaged, and the result holds. For a graph with radius at least $2$, suppose the cop initially occupies some vertex $x$.  Let $z$ be a vertex distance rad$(G)$ from $x$ and let $P$ be the shortest $xz$-path.  If the robber initially occupies a vertex $y$ on $P$ that is distance two from $x$, then the robber can damage at least rad$(G)-1$ vertices by moving along $P$ towards $z$. 

\begin{theorem}[\cite{throttling}]
\label{thm:newlowerbound}
For any graph $G$, ${\rm dmg}(G) \geq {\rm rad}(G)-1$. 
\end{theorem}

\section{The damage number of products}\label{sec:products}

Products are powerful tools in graph theory. Interesting relationships often emerge between the value of a graph parameter of the input graphs versus their product and for this reason, we  consider the damage number of the Cartesian product of graphs. The \emph{Cartesian product} of graphs $G$ and $H$, written $G\square H$, is the graph with vertex set $V(G) \times V(H)$ specified by putting $(u,v)$  adjacent to $(u',v')$ if and only if (1) $u=u'$ and $v$ is adjacent to $v'$ in $H$ or (2) $v=v'$ and $u$ is adjacent to $u'$ in $G$. 
 
 We begin with a simple observation that implies $G \square H$ is not copwin, regardless of the properties of $G$ and $H$, and then conclude that there is always a strategy for the robber to damage ${\rm dmg}(G \square H)$ many vertices on $G \square H$ and avoid capture.  

\begin{proposition}\label{obs:nocorner} If $G$ and $H$ are graphs, then $G \square H$ contains no c-dominated vertices. 
 \end{proposition}

\begin{proof}For a contradiction, suppose $(x,y)$ is a c-dominated vertex in $G \square H$.  Then there exists some vertex $(u,v) \in V(G \square H)$ such that $N[(x,y)] \subseteq N[(u,v)]$.  Without loss of generality, suppose $u$ is adjacent to $x$ in $G$ and $y=v$ in $H$.  Since $H$ has no isolated vertices, there exists $z \in V(H) \backslash \{y\}$ that is adjacent to $y$.  Then vertex $(x,z)$ exists and is adjacent to $(x,y)$ in $G \square H$.  But $(u,y)$ is not adjacent to $(x,z)$ in $G \square H$ because $y \neq z$ and $u \neq x$.  Consequently, $(x,z) \in N[(x,y)]$ but $(x,z) \not\in N[(u,v)]$ and we have a contradiction. \end{proof}

Observe that Proposition~\ref{obs:nocorner} coupled with Observation~\ref{obs:no-c} implies there exists a strategy for the robber to damage ${\rm dmg}(G \square H)$ vertices and avoid capture.  We assume throughout this section that the robber will be using such a strategy on $G \square H$.

\subsection{Products of general graphs}\label{Sec2p1}

It is convenient to use cop strategies for games on $G$ and $H$ to create a cop strategy on $G \square H$.  The difficulty is that any move by the cop on $G \square H$ corresponds to a cop move on either $G$ or $H$.  If the cop's first move is to change their first coordinate on $G \square H$, we observe this corresponds to the cop moving in the game played on $G$.  Suppose the robber's first move is to change their second coordinate on $G \square H$; this corresponds to the robber moving in the game played on $H$.  However, in terms of the subgame on $H$, it corresponds to the situation on $H$ where the cop has initially passed and the robber moves first.  As an example, from~\cite{CS2019}, we know ${\rm dmg}(C_m) = \lfloor \frac{m-1}{2}\rfloor$ and therefore ${\rm dmg}(C_6)=2$.  However, if the robber were allowed to move first, the robber could initially occupy a vertex adjacent to the cop and damage ${\rm dmg}(C_6)+1=3$ vertices.  In contrast, if the robber were allowed to move first on $C_5$, they would still only be able to damage ${\rm dmg}(C_5)=2$ vertices.  We will explore cycles in depth in Section~\ref{subsec:cycles}, but this motivates the following question. 

\begin{question} For which graphs $H$ can the cop pass during the first round and still prevent the robber from damaging more than ${\rm dmg}(H)$ vertices? \end{question}

Let $\rm{dmg}'(H)$ denote the minimum number of vertices damaged on graph $H$, when the cop passes during the first round. We can easily see that  \begin{equation*} {\rm dmg}'(H) \in \big\{ {\rm dmg}(H), {\rm dmg}(H)+1\big\}.\end{equation*}

Let $H$ be a graph and $x$ be a vertex such that if the cop initially occupies $x$, then the cop can prevent the robber from damaging more than ${\rm dmg}(H)$ vertices.  Suppose the robber initially occupies some vertex $y \in V(H)$ where $y \neq x$.  The cop initially passes and the robber damages $y$ and moves to $z$ (it is possible that $y=z$).  From the positions of $x$ and $z$, the cop will move next and play proceeds as normal, leaving the robber unable to damage more than ${\rm dmg}(H)$ vertices in addition to $y$. Thus, ${\rm dmg}'(H) \leq {\rm dmg}(H)+1$.  Certainly ${\rm dmg}'(H) \geq {\rm dmg}(H)$ because the robber could adopt the same strategy and also initially pass on their first turn, then they damage at least ${\rm dmg}(H)$ vertices. 

\begin{lemma}\label{lem:Phase2} Let $G$ and $H$ be graphs.  In $G \square H$, suppose that during some round, the cop and robber occupy vertices with the same first coordinate and it is the robber's turn to move.  The robber can damage at most an additional ${\rm dmg}'(H)|V(G)|$ vertices of $G \square H$.\end{lemma}

\begin{proof} In $G \square H$, suppose the cop and robber initially occupy vertices with the same first coordinate and it is the robber's turn to move.  Let $|V(G)|=m$ and observe that graph $G \square H$ has $m$ vertex-disjoint subgraphs, denoted $H_1,H_2,\dots,H_m$ such that $H_i \cong H$ and $\cup_{i=1}^m V(H_i)=V(G \square H)$.  Then every vertex in $H_i$ has the same first coordinate and so the cop and robber initially both occupy vertices in $H_i$ for some $i \in \left\{1, \dots,m\right\}$.  We refer to $H_1,H_2,\dots,H_m$ as {\it copies} of $H$.

The cop's strategy is as follows: whenever the robber changes their first coordinate, the cop changes their first coordinate to match.  This will ensure that after the cop moves at each round, the cop and robber will occupy the same copy of $H$.  Whenever the robber changes their second coordinate, the cop changes their second coordinate, following a cop strategy that minimizes the damage on $H$. 

At worst, every time the robber changes their second coordinate (i.e. moves within a copy of $H$), the robber could then change their first coordinate $m$ times (and damage $|V(G)|$ many vertices) before changing the second coordinate again.  This results in at most $({\rm dmg}'(H))\cdot(m)$ vertices damaged.  We note that since the cop's strategy is to always change the same coordinate the robber previously changed, the robber will be the first to change their second coordinate (i.e. move within a copy of $H$) and thus, can damage ${\rm dmg}'(H)$, rather than ${\rm dmg}(H)$ many vertices in each copy of $H$. \end{proof} 

Naturally, it follows from Lemma~\ref{lem:Phase2} that if the cop and robber occupy vertices with the same second coordinate and it's the robber's turn to move, then the robber can damage at most ${\rm dmg}'(G)|V(H)|$ additional vertices of $G \square H$. However, for both this and later results in this subsection, we do not formally state the symmetric result. 

\begin{theorem}\label{thm:oneuni} Let $G$ and $H$ be graphs. If both $G$ and $H$ have universal vertices, then \[{\rm dmg}(G \square H) \leq \min \{|V(G)|,|V(H)|\}.\]  
\end{theorem}

\begin{proof}
Suppose $G$ and $H$ have universal vertices $u$ and $v$, respectively.  The cop initially occupies $(u,v)$ in $G \square H$.  If the cop initially changes their first coordinate to match that of the robber, then by Lemma~\ref{lem:Phase2}, the robber can damage ${\rm dmg}'(H)|V(G)|$ vertices.  Note that ${\rm dmg}'(H)=1$ since $H$ has a universal vertex.  Thus, the robber can damage at most $|V(G)|$ vertices.  By a similar argument, if the cop initially changes their second coordinate to match that of the robber, the robber can then damage at most $|V(H)|$ vertices.
\end{proof}

\begin{theorem}\label{thm:prodK} Let $K_m$ and $K_n$ be complete graphs on $m$ and $n$ vertices, respectively.  If $n \geq m \geq 3$ then ${\rm dmg}(K_m \square K_n) = m$.\end{theorem}

\begin{proof} The upper bound follows from Theorem~\ref{thm:oneuni}.  For a lower bound, we suppose that ${\rm dmg}(K_m \square K_n) = k$, where $k  < m$ and obtain a contradiction.  By Observation~\ref{obs:no-c}, we can assume the robber uses a strategy to damage $k$ vertices and avoid capture.  By some round $t$, the robber has damaged $k$ vertices and the robber cannot move to an undamaged vertex without being captured. Suppose that during round $t$, the cop occupies $(x,y)$, the robber occupies $(u,v)$ and it is the robber's turn to move. When the robber moves, $(u,v)$ is damaged, so to obtain a contradiction, we need only observe that $(u,v)$ has at least $k$ neighbours that are not also adjacent to $(x,y)$.     

If $x=u$, observe $$\Big| N[(u,v)]\backslash N[(u,y)]\Big| = \Big| \{ (z,v): z \in V(K_m) \}\Big| = m-1 \geq k.$$

If $x \neq u$ and $y \neq v$, observe \begin{eqnarray*}\Big| N[(u,v)] \backslash N[(x,y)]\Big| &=& \Big| \big\{ (u,z): z \in V(K_n)\backslash \{y\}\big\} \cup \big\{ (z,v): z \in V(K_m)\backslash \{x\}\big\} \Big|\\ &=& m+n-3 \geq k.\end{eqnarray*} Thus, the robber can move to an undamaged vertex that is not adjacent to the cop and, regardless of the cop's next move, the robber will damage a $(k+1)^{th}$ vertex.\end{proof}

From Theorem~\ref{thm:prodK}, we see that the bound provided in Theorem~\ref{thm:oneuni} is exact for some graphs.  We note, however, that this is not always the case.  If $u,v$ are universal vertices in $K_{1,m},K_{1,n}$ respectively, a cop can occupy vertex $(u,v)$ in $K_{1,m} \square K_{1,n}$ which o-dominates every vertex to which it is not adjacent.  Consequently, ${\rm dmg}(K_{1,m} \square K_{1,n})=1$.     
We next consider arbitrary graphs $G$ and $H$.

\begin{theorem}\label{thm:upperH} For any graphs $G$ and $H$, $${\rm dmg}(G \square H) \leq \max \Big\{ {\rm dmg}(G)|V(H)|, {\rm dmg}'(H)|V(G)|\Big\}.$$ \end{theorem}

\begin{proof} Let $G$ be a graph on $m$ vertices, labeled $u_0,\dots,u_{m-1}$.  Without loss of generality, suppose that by occupying $u_0$ and moving first, as usual, the cop can prevent more than ${\rm dmg}(G)$ vertices from being damaged.  Let $H$ be a graph on $n$ vertices, labeled $v_0,\dots,v_{n-1}$.  Suppose that by occupying $v_0$ and passing during the first round, the cop can prevent more than ${\rm dmg}'(H)$ vertices from being damaged. In $G \square H$, the cop will initially occupy $(u_0,v_0)$.

On $G \square H$, the cop's strategy will often be to mirror optimal cop moves of a game played on $G$ or $H$.  
 
 In $G \square H$, suppose $u_p$ and $u_x$ are the first coordinates of the cop and robber's positions, respectively during some round.   On $G$, suppose the cop and robber occupy $u_p$ and $u_x$ (respectively) and an optimal strategy for the cop in $G$ is to next move to $u_\alpha$. In $G \square H$, if the cop then changes first coordinate to $u_\alpha$, then we say {\it the cop changes their first coordinate according to an optimal cop-strategy for $G$}.

Suppose the cop occupies $(u_0,v_0)$, the robber occupies $(u_x,v_y)$, and it is the cop's turn to move. During the first round (i.e. when $t=1$), the cop changes the first coordinate in $G \square H$ according to an optimal cop-strategy for $G$.  If this results in the cop matching the first coordinate of the robber, then by Lemma~\ref{lem:Phase2}, the robber can damage at most ${\rm dmg}'(H)|V(G)|$ vertices on $G \square H$.  Otherwise, for $t > 1$,

\begin{enumerate}[(1)] 

\item If the cop can move during round $t$ to match first or second coordinate to that of the robber, then the cop will do so.

\item Otherwise, 

\begin{enumerate}
	\item if the robber changed their first coordinate during round $t-1$, then during round $t$, the cop changes their first coordinate according to an optimal cop-strategy on $G$.
	
	\item if the robber changed their second coordinate during round $t-1$, then during round $t$, the cop changes their second coordinate according to an optimal cop-strategy on $H$.
	
	\item if the robber passed during round $t-1$, the cop will pass during round $t$.  
\end{enumerate}
\end{enumerate}

Let $D(G)$ be a set of ${\rm dmg}(G)$ vertices that can be damaged by the robber on graph $G$ if the cop and robber initially occupy $u_0$ and $u_x$, respectively.  Let $D'(H)$ be a set of ${\rm dmg}'(H)$ vertices that can be damaged by the robber on graph $H$ if the cop and robber initially occupy $v_0$ and $v_y$, respectively; and the cop initially passes.

Let $$S_G = \Big\{ (u_i,v_j) : u_i \in D(G), j \in \{0,1,\dots,n-1\} \Big\}$$ and $$S'_H = \Big\{ (u_i,v_j) : i \in \{0,1,\dots,m-1\}, v_j \in D'(H) \Big\}$$ and observe that $(u_x,v_y) \in S_G \cap S'_H.$ Note that the strategy of (2) restrict the robber to vertices in $S_G \cup S'_H$.  If the robber only ever occupies vertices of $S_G \cap S'_H$, then the result holds because $$\big|S_G \cap S'_H\big| \leq \max\Big\{ |S_G|, |S'_H|\Big\} \leq \max \Big\{ {\rm dmg}(G)|V(H)|, {\rm dmg}'(H)|V(G)|\Big\}.$$

Suppose that during some round, the robber moves to a vertex of $S_G \backslash S'_H$; observe that to do this, the robber changes their second coordinate.  The cop is then able to move to match the second coordinate of the robber.  Observe that by (1) and (2), the robber will be captured if the robber moves to a vertex outside $S_G$.  Thus, the robber can damage at most $|S_G| = {\rm dmg}(G)|V(H)|$ vertices and the result holds.  Similarly, if the robber instead moves from a vertex in $S_G \cap S'_H$ to a vertex in $S'_H \backslash S_G$, then the cop will move to match the first coordinate of the robber.  Then again, the robber will be captured if the robber moves to a vertex outside $S'_H$.  Consequently, the robber can damage at most $|S'_H| = {\rm dmg}'(H)|V(G)|$ vertices and the result holds.\end{proof} 

An illustration of sets $S_G$ and $S'_H$ is given in Figure~\ref{fig:cycles} for $C_{10} \square C_{13}$. Finally, we restate Theorem~\ref{thm:upperH} in terms of the damage numbers of $G$ and $H$.

\begin{figure}[htbp]
\[ \includegraphics[width=0.725\textwidth]{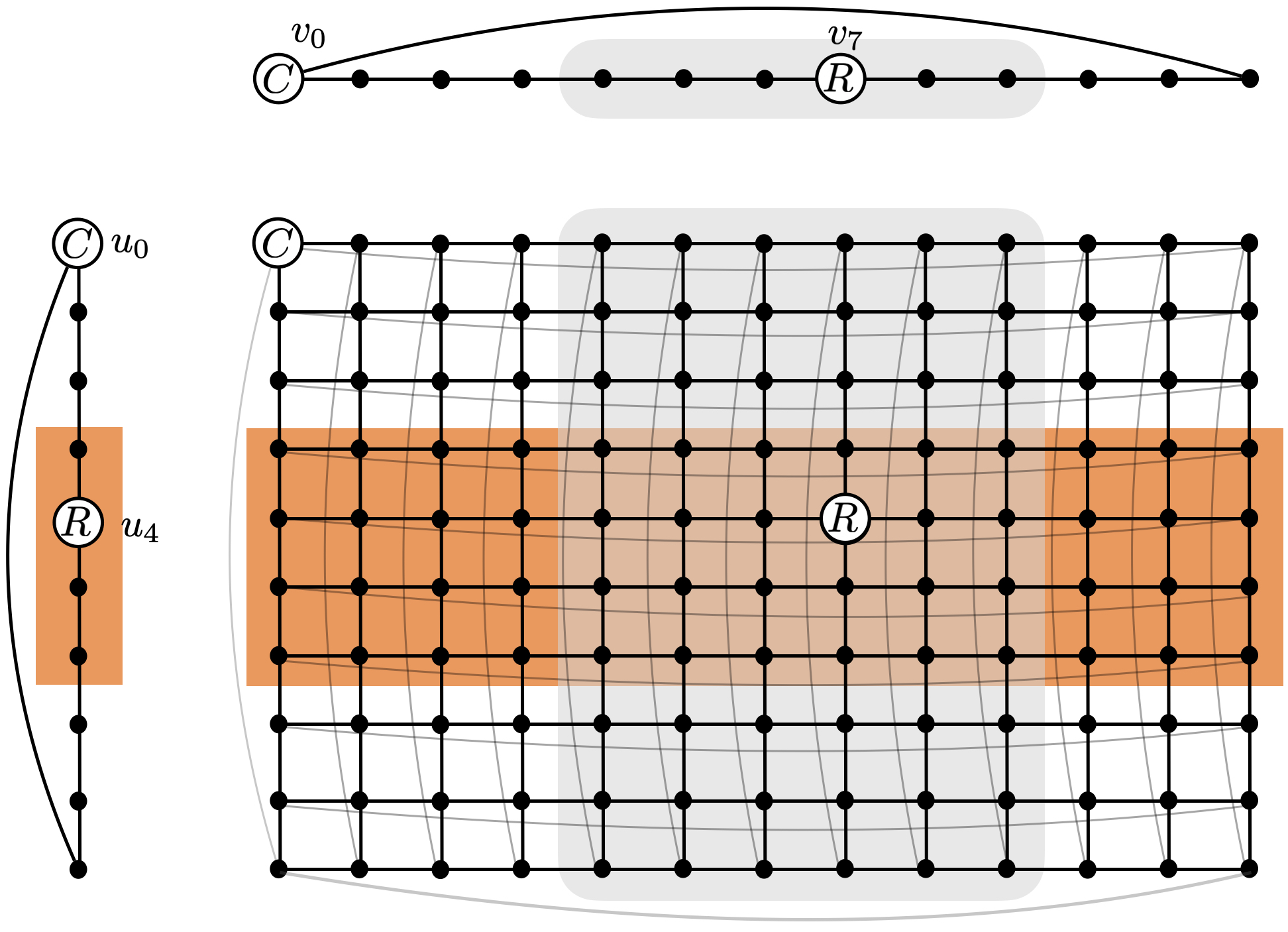} \]

\caption{An example of the set $D(C_{10})$ when $p=0$ and $x=4$ (rectangular region in left graph); $D'(C_{13})$ when $q=0$ and $y=7$ (grey region in top graph); and the corresponding sets $S_{C_{10}}$ (rectangular region in $C_{10} \square C_{13}$) and $S'_{C_{13}}$ (grey region in $C_{10} \square C_{13}$).}

\label{fig:cycles}

\end{figure}

\begin{corollary}\label{corcor} For any graphs $G$ and $H$, $${\rm dmg}(G \square H) \leq \max \Big\{ {\rm dmg}(G)|V(H)|, \big({\rm dmg}(H)+1\big)|V(G)|\Big\}.$$ \end{corollary}

\subsection{Products of trees}\label{subsec:trees}

We begin by determining the damage number of a tree and then introduce some notation and terminology used to prove ${\rm dmg}(T \square T')={\rm rad}(T \square T')-1$ for any finite trees $T$ and $T'$.

\begin{corollary}\label{cor:tree} For any tree $T$, ${\rm dmg}(T) = {\rm rad}(T)-1$.\end{corollary}

\begin{proof}The lower bound follows from Theorem~\ref{thm:newlowerbound}.  For the upper bound, it suffices to show that ${\rm capt}(T)\leq {\rm rad}(T)$ as Cox and Sanaei~\cite{CS2019} pointed out that ${\rm dmg}(G) \leq {\rm capt}(G)-1$ for any copwin graph $G$.  Suppose the cop initially occupies a center $c$ of tree $T$.  If the robber initially occupies vertex $u$, the robber will be distance at most ${\rm rad}(T)$ away from the cop. The cop can then take the shortest path from $c$ to $u$ to capture the robber, which will take ${\rm rad}(T)$ or fewer rounds. Therefore ${\rm capt}(T) \leq {\rm rad}(T)$.\end{proof}  

\begin{definition} Label the vertices of trees $T$ and $T'$ as $V(T) = \{u_0,u_1,\dots,u_{|V(T)|-1}\}$ and $V(T') = \{v_0,v_1,\dots,v_{|V(T')|-1}\}$.  If the cop occupies $(u_p,v_q)$ and the robber occupies $(u_i,v_j)$ then

\begin{enumerate}
\item[(1)] the {\it cop-robber distance} in the first and second coordinates is ${\rm dist}_T(u_p,u_i)$ and ${\rm dist}_{T'}(v_q,v_j)$, respectively;

\item[(2)] the {\it cop-robber distance is smaller in the first coordinate} if ${\rm dist}_T(u_p,u_i) < {\rm dist}_{T'}(v_q,v_j)$; 

\item[(3)] the {\it cop-robber distance is smaller in the second coordinate} if ${\rm dist}_T(u_p,u_i) > {\rm dist}_{T'}(v_q,v_j)$; 

\item[(4)] the {\it equidistant set} of $(u_p,v_q) \in V(T \square T')$ is the set of vertices $$\big\{ (u_k,v_\ell):{\rm dist}_T(u_p,u_k)={\rm dist}_{T'}(v_q,v_\ell)\big\}.$$ \end{enumerate} If the cop occupies $(u_p,v_q)$ and the robber occupies a vertex in the equidistant set of $(u_p,v_q)$, then {\it the cop and robber are equidistant} in $T \square T'$. \end{definition} 

The capture time in $T$ and $T'$ respectively will be a helpful tool to prove a result in $T\square T'$. Hence, we call these relative capture times for the input graphs of the product. Formally, we have the following.
\begin{definition} We denote by ${\rm capt}_T(u_p;u_i)$, the {\it relative capture time} on $T$ when the cop initially occupies $u_p$ and the robber initially occupies $u_i$.  \end{definition}

\begin{theorem}\label{thm:prodTrees} For any finite trees $T$ and $T'$, ${\rm dmg}(T \square T') = {\rm rad}(T \square T')-1.$ \end{theorem}

\begin{proof}

The lower bound follows from Theorem~\ref{thm:newlowerbound}.  Suppose the cop initially occupies $(u_0,v_0)$ where $u_0$ and $v_0$ are center vertices for $T$ and $T'$, respectively.  We present a strategy for the cop in which at most ${\rm rad}(T \square T')-1$ vertices can be damaged.\medskip

\noindent{\it Cop-strategy:} Suppose the cop occupies $(u_a,v_b)$, the robber occupies $(u_i,v_j)$, and it is the cop's turn to move.  
\begin{enumerate}[(1)]

\item If the cop and robber are equidistant, then the cop passes.  

\item Otherwise, 

(i) if ${\rm dist}_T(u_a,u_i) > {\rm dist}_{T'}(v_b,v_j)$, the cop changes their first coordinate in accordance with an optimal cop-strategy on $T$; 

(ii) if ${\rm dist}_T(u_a,u_i) < {\rm dist}_{T'}(v_b,v_j)$, the cop changes their second coordinate in accordance with an optimal cop-strategy on $T'$. 

\end{enumerate}
 
If the cop follows the above cop-strategy, at some round the cop and robber will be equidistant since $T$ and $T'$ are finite.  When the cop and robber are equidistant for the first time, {\it Phase 1} is complete.  We next bound the number of rounds before Phase 1 finishes, which bounds the number of vertices that have been damaged when Phase 1 finishes.

Recall the cop is initially located at $(u_0,v_0)$ and suppose the robber is initially located at $(u_i,v_j)$.  If the cop and robber are initially equidistant, then Phase 1 is complete before the cop moves (and no vertices are damaged in Phase 1).  Otherwise, without loss of generality, suppose ${\rm dist}_T(u_0,u_i) > {\rm dist}_{T'}(v_0,v_j)$; so the cop-robber distance is smaller in the second coordinate.  Note that after each move of the cop (following (2)(i)) and after each move of the robber, either the cop and robber are equidistant or the cop-robber distance remains smaller in the second coordinate.  Consequently, such a cop only ever changes first coordinate in Phase 1.   Suppose the cop and robber become equidistant during round $t$ when the cop occupies $(u_a,v_0)$ and the robber occupies $(u_k,v_\ell)$.  Then $t = {\rm dist}_T(u_0,u_a)$ and ${\rm dist}_T(u_a,u_k) = {\rm dist}_{T'}(v_0,v_\ell)$.

{\it Phase 2} begins immediately after the cop and robber become equidistant. Given (1) the cop-strategy, we can assume it is the robber's turn to move.   If the robber passes, then the cop will continue to pass until the robber moves; thus, we need only consider active moves by the robber:  

\begin{enumerate}[(a)]

\item Suppose the robber's move increases the cop-robber distance in the first coordinate or decreases the cop-robber distance in the second coordinate.  Then the cop-robber distance will be smaller in the second coordinate.  Following (2)(i), the cop's move will decrease the cop-robber distance in the first coordinate and leave the cop and robber equidistant.

\item Suppose the robber's move increases the cop-robber distance in the second coordinate or decreases the cop-robber distance in the first coordinate.  Then the cop-robber distance will be smaller in the first coordinate.  Following (2)(ii), the cop's move will  decrease the cop-robber distance in the second coordinate and leave the cop and robber equidistant.
\end{enumerate}

In (a) and (b), each move by the cop leaves the cop and robber equidistant; however, it is important to observe that 
the relative capture time of at least one of $T,T'$ decreases during each round in Phase 2.  Recall that when Phase 2 begins, the robber occupies $(u_k,v_\ell)$.  After at most $$\big({\rm capt}_{T'}(v_0;v_\ell)-1\big)+\big({\rm capt}_T(u_a;u_k)-1\Big)$$
rounds of Phase 2, the cop and robber are equidistant, the vertices corresponding to the first and second coordinate of the cop and robber are stems and leaves in $T$ and $T'$; i.e. if a vertex occupied by the cop dominates a vertex occupied by the robber in $T$ and in $T'$, then in the Cartesian product, $T\square T'$, the cop will $o$-dominate the vertex occupied by the robber.  The robber can damage no new vertices, apart from $(u_k,v_\ell)$.  By the proof of Corollary~\ref{cor:tree}, ${\rm capt}_{T'}(v_0;v_\ell) \leq {\rm rad}(T')$.   Note that ${\rm dist}_T(u_0,u_a)+{\rm capt}_T(u_a;u_k) \leq {\rm rad}(T)$.  Thus, after at most \begin{eqnarray*} \big({\rm capt}_T(u_a;u_k)-1+{\rm dist}_{T}(u_0,u_a)\big)+\big({\rm capt}_T(v_0;v_\ell)-1\big) &\leq& {\rm rad}(T)-1+{\rm rad}(T')-1 \\ &\leq& {\rm rad}(T)+{\rm rad}(T')-2 \\ &=& {\rm rad}(T \square T')-2 \end{eqnarray*} rounds, the cop occupies a vertex $(u_m,v_n)$ such that $u_m$ and $v_n$ are stems in $T$ and $T'$ respectively and the robber occupies $(u_x,v_y)$ such that $u_x$ and $v_y$ are leaves in $T$ and $T'$ respectively where $u_x \in N_T(u_m)$ and $v_y \in N_{T'}(v_n)$.  Since the robber can remain on $(u_x,v_y)$ indefinitely, the robber also damages $(u_x,v_y)$.  Thus, at most ${\rm rad}(T \square T')-2+1 = {\rm rad}(T \square T')-1$ vertices are damaged.\end{proof}  

For both trees and the product of trees, the damage number is one less than the radius of the graph.  Consequently, we ask which other graphs have this property.

\begin{question} For which graphs $G$ is ${\rm dmg}(G) = {\rm rad}(G)-1$?  For which graphs $G$ and $H$ is ${\rm dmg}(G \square H) = {\rm rad}(G \square H)-1$? \end{question}

\subsection{Products of cycles}\label{subsec:cycles}

In this section, for every cycle $C_m$, we will assume the vertices are labeled as $u_0,\dots,u_{m-1}$ where $u_i \sim u_{i+1}$ (mod $m$).  Similarly, for every cycle $C_n$, the vertices are labeled $v_0,\dots,v_{n-1}$ where $v_i \sim v_{i+1}$ (mod $n$).
 
\begin{theorem}\label{thm: cycle_shadow}
 If $m,n \geq 4$ then \[{\rm dmg}(C_m \square C_n) \geq \max \Big\{{\rm dmg}(C_m) |V(C_n)|, {\rm dmg}(C_n)|V(C_m)|\Big\}.\]
\end{theorem}

\begin{proof}
We must show that the robber can damage $\max \Big\{{\rm dmg}(C_m) |V(C_n)|, {\rm dmg}(C_n)|V(C_m)|\Big\}$ vertices regardless of the cop strategy. 
In order to discuss the location of the robber, relative to the cop, consider the copy of $C_m$ in which the cop initially places themselves as copy $0$. The index for the second coordinate is $v_0$ (on $C_n$). We give the following robber strategy. The robber starts in the same copy of $C_m$ as the cop, at distance two away. Without loss of generality, suppose the cop started at $(u_i, v_0)$, then the robber starts at $(u_{i+2}, v_0)$. The robber responds to each possible \emph{initial} cop move using the following strategy:
\begin{enumerate}[(1)]
\item If the cop passes, then the robber increases the index of their first coordinate mod~$m$ (thus moves in their copy of $C_m$), maintaining a distance of at least two from the cop.

\item If the cop increases their distance to the robber in the first coordinate, then the robber decreases their distance to the cop in the first coordinate.

\item If the cop decreases their distance to the robber in the first coordinate, then the robber increases their distance to the cop in the first coordinate.

\item If the cop increases their distance in the second coordinate, then the robber plays as though the cop is still occupying the same vertex in the robber's copy of $C_m$, and treats this as a pass move by the cop. The robber will damage a new vertex as they move away from the cop in the first coordinate. This will be called the \emph{shadow strategy}. 
\end{enumerate}

More precisely, the shadow strategy for the robber is to pretend that the cop is occupying a vertex of the robber's copy of $C_m$ by projecting the image of the cop from the second coordinate, to its corresponding position in the robber's copy of $C_m$: the cop's shadow. For example, if the robber is occupying the vertex labelled $(u_{i+2}, v_j)$, and the cop is occupying $(u_{i}, v_{j+1})$, the robber will play as though the cop is occupying $(u_i, v_j)$.  The robber plays in this way to maintain a distance of at least two from the cop in the robber's copy of $C_m$. When the robber moves to a new copy of $C_m$, this strategy prevents the robber from being immediately captured. If the cop moves in the second coordinate, this works to the advantage of the robber as long as the robber plays the shadow strategy.

 The robber damages at least ${\rm dmg}(C_m)$ vertices in copy 0 as they can guarantee this amount when the cop plays optimally. 
 If the cop is elsewhere in the graph, but the robber maintains distance at least two from the cop's shadow, then the cop will never be able to capture the robber, nor will they be able to prevent the robber from changing copies of $C_m$. If the cop returns to a previously visited vertex other than their starting position or passes, then the robber moves to the closest undamaged vertex that is at least distance two away from the cop in the same copy of $C_m$. 
If such a vertex does not exist, it means the robber has damaged ${\rm dmg}(C_m)$ vertices in copy 0. The robber instead increases their second coordinate in order to begin the process again, but in a different copy of $C_m$. The robber applies this strategy until all copies of $C_m$ have been visited, and ${\rm dmg}(C_m)$ vertices have been damaged in each copy.\end{proof}

In~\cite{CS2019}, the authors provide a strategy for the cop to restrict the damage on $C_m$ to at most $\lfloor (m-1)/2\rfloor$ vertices.  Informally, the cop always moves in the opposite direction to the robber's last move.  We next show that if the cop initially passes, the damage number only increases if $m$ is even.  Lemma~\ref{lem:prime} will be then used to prove an upper bound for ${\rm dmg}(C_m \square C_n)$.

\begin{lemma}\label{lem:prime} For any $m \geq 4$, ${\rm dmg}'(C_m) = \begin{cases} {\rm dmg}(C_m) & \text{~if $m$ is odd;} \\ {\rm dmg}(C_m)+1 & \text{~if $m$ is even.} \end{cases}$\medskip 

\noindent Furthermore, in the case where $m$ is even, if the robber begins on a vertex of $C_m$ that is not adjacent to the cop and the cop passes during the first round, the cop can prevent the robber from damaging more than ${\rm dmg}(C_m)$ vertices.\end{lemma}

\begin{proof} Recall from Section~\ref{Sec2p1} that ${\rm dmg}'(C_m) \in \{{\rm dmg}(C_m),{\rm dmg}(C_m)+1\}$, so it remains to show that ${\rm dmg}'(C_m) \geq {\rm dmg}(C_m)+1$ when $m$ is even and ${\rm dmg}'(C_m) \leq {\rm dmg}(C_m)$ when $m$ is odd.

Suppose $m$ is even.  We provide a robber strategy which will always damage ${\rm dmg}(C_m)+1$ many vertices, regardless of the cop strategy.  Suppose the cop and robber initially occupy vertices $u_0$ and $u_1$, respectively. The robber's strategy is to increase their index at each round, until they occupy $u_{m/2}$.

Since the cop initially passes, there is no cop strategy to prevent the robber from moving to $u_{m/2}$ before the cop can move to $u_{m/2+1}$ or $u_{m/2-1}$. Thus, the robber will be able to damage $u_1,u_2,\dots, u_{m/2}$ and damage at least ${\rm dmg}(C_m)+1$ vertices.

We next allow $m$ to be either odd or even and provide a cop strategy to find an upper bound for ${\rm dmg}'(C_m)$ while simultaneously proving the second statement in the theorem.  Suppose the cop and robber initially occupy $u_0$ and $u_i$ (respectively) and without loss of generality, $1 \leq i \leq \frac{m}{2}$.  The cop initially passes and for $t >1$, the cop moves according to the following cycle-strategy, where indices are taken modulo $m$: 

\begin{enumerate}[(1)]\item If the robber increases (decreases) their index during round $t-1$, then the cop decreases (increases) their index during round $t$;  

\item If the robber passes during round $t-1$, then the cop passes during round $t$, unless the cop is adjacent to the robber.  In this case, the cop moves to capture the robber.\end{enumerate}

Given (2), it suffices to consider only ``active rounds" of the robber; that is, rounds during which the robber does not pass.  For the remainder of the proof, all rounds will be considered to be active rounds.

First, we assume $i>1$ and show that by following the cycle-strategy, the cop can prevent the robber from damaging a vertex outside the set $$S_i = \big\{ u_k, u_{k+1},u_{k+2},\dots,u_{\lfloor \frac{m+i-1}{2}\rfloor} \big\}$$ where $k = \lceil \frac{i+1}{2}\rceil$.  Suppose that during some round $t>1$, the robber moves from $u_{k+1}$ to $u_k$.  We count the number of  rounds that have passed since the robber first occupied $u_i$ to when the robber moves to $u_k$: for some non-negative integer $x$, during a total of $x$-many rounds (not necessarily consecutive), the robber increased their index (mod $m$) and during $(x+\lfloor \frac{i-1}{2}\rfloor)$-many rounds, the robber decreased their index (mod $m$). Thus, after $2x+\lfloor \frac{i-1}{2}\rfloor$ rounds, the robber occupies $u_k$; i.e. the vertex with subscript $$i+x-\Big(x+\Big\lfloor \frac{i-1}{2}\Big\rfloor\Big) = \Big\lceil \frac{i+1}{2}\Big\rceil = k.$$  Since the cop follows the cycle-strategy, after $2x+\lfloor \frac{i-1}{2}\rfloor+1$ rounds (where the robber has not yet moved in round $2x+\lfloor \frac{i-1}{2}\rfloor$), the cop occupies the vertex with subscript $0-x+(x+\lfloor \frac{i-1}{2}\rfloor) = \lfloor \frac{i-1}{2}\rfloor$.  Observe that $$\Big\lfloor \frac{i-1}{2}\Big\rfloor = \begin{cases} k-1 & \text{~if $i$ is odd} \\ k-2 & \text{~if $i$ is even}.\end{cases}$$ 

At this point, the robber occupies $u_k$, the cop occupies $u_{k-1}$ or $u_{k-2}$ and it the robber's turn to move during round $2x+\lfloor \frac{i-1}{2}\rfloor$.  If $i$ is odd, the cop and robber are adjacent and if $i$ is even, the cop and robber are distance two apart.  In either case, the robber cannot move to $u_{k-1}$.  By a similar argument, the robber cannot move to a higher indexed vertex than $u_{\lfloor \frac{m+i-1}{2}\rfloor}$ without being captured.  Thus, if $i>1$, at most $|S_i| = \lfloor \frac{m-1}{2}\rfloor= {\rm dmg}(C_m)$ vertices are damaged.

Second, we assume $i=1$ and show that by following the cycle-strategy, the cop can prevent the robber from damaging a vertex outside the set $S_1 = \{u_1,u_2,\dots,u_{\lfloor m/2\rfloor}\}.$ Suppose that during some round $t>1$, the robber moves from $u_{\lfloor m/2\rfloor -1}$ to $u_{\lfloor m/2\rfloor}$.  Then during a total of $x$-many rounds (not necessarily consecutive), the robber decreased their index and during $(x+\lfloor \frac{m}{2}\rfloor -1)$-many rounds, the robber increased their index, leaving the robber at $u_{\lfloor m/2\rfloor}$.  After the cop has moved during round $2x+\lfloor \frac{m}{2}\rfloor$,  the robber occupies $u_{\lfloor m/2 \rfloor}$,  the cop occupies $u_{m-(\lfloor\frac{m}{2}\rfloor-1)}$ and it is the robber's turn to move.  But as the cop and robber are either adjacent (if $m$ is even) or distance $2$ apart (if $m$ is odd), the robber cannot move to $u_{\lfloor m/2\rfloor+1}$.  By a similar argument, the robber cannot move to a lower indexed vertex than $u_1$.  Thus, if $i=1$, the number of vertices damaged is at most $$|S_1| = \Big\lfloor \frac{m}{2} \Big\rfloor \leq \begin{cases} {\rm dmg}(C_m) & \text{if $m$ is odd} \\ {\rm dmg}(C_m)+1 & \text{if $m$ is even.}\end{cases}$$\end{proof}

The next result follows from Theorem~\ref{thm:upperH}, Theorem~\ref{thm: cycle_shadow}, and Lemma~\ref{lem:prime}.

\begin{corollary}\label{cor:mostcycles} For $m,n \geq 4$ where at least one of $m, n$ is odd,  $${\rm dmg}(C_m \square C_n) = \max\Big\{{\rm dmg}(C_m)|V(C_n)|, {\rm dmg}(C_n)|V(C_m)|\Big\}.$$\end{corollary} 

Corollary~\ref{cor:mostcycles} provides a family of graphs for which the upper bound of Theorem~\ref{thm:upperH} is exact.  Unfortunately, the bound of Theorem~\ref{thm:upperH} is not exact when both $m$ and $n$ are even.  

\begin{lemma}\label{thm: upperbound product of cycles} If $m, n \geq 4$ with $m$ and $n$ both even, then $${\rm dmg}(C_m \square C_n) \leq 1+\max\Big\{{\rm dmg}(C_m)|V(C_n)|, {\rm dmg}(C_n)|V(C_m)|\Big\}.$$\end{lemma}

\begin{proof}
Suppose the cop initially occupies $(u_0,v_0)$. Let $$S = \{ (u_1,v_1), (u_1,v_{n-1}), (u_{m-1},v_1), (u_{m-1},v_{n-1})\}.$$  We consider two cases for the initial position of the robber, either the robber begins at a vertex in $S$ or a vertex in $V(C_m \square C_n) \backslash S$.  

Case 1. Suppose the robber initially occupies a vertex in $V(C_m \square C_n) \backslash S$.  Considering the corresponding positions of the cop and robber in $C_m$ and $C_n$, note that the cop is not adjacent to the robber in at least one of $C_m$ or $C_n$.  Using Lemma~\ref{lem:prime} and the cop-strategy given in the proof of Theorem~\ref{thm:upperH}, the cop can limit the damage to at most $\max\{ {\rm dmg}(C_m)|V(C_n)|,  {\rm dmg}(C_n)|V(C_m)|\}$ vertices.    

Case 2. Suppose the robber initially occupies a vertex in $S$ and without loss of generality, suppose the robber occupies $(u_1,v_1)$. The cop initially passes, the robber damages $(u_1,v_1)$, then the robber moves to $(u_2,v_1)$ or $(u_1,v_2)$.  (If the robber moves to a different neighbour of $(u_1,v_1)$, the cop will next capture the robber and if the robber passes, the cop will continue to pass until the robber moves to a new vertex.)  Without loss of generality, suppose the robber moved to $(u_2,v_1)$.  

Consider an instance of the game where the cop begins at $(u_0,v_0)$, the robber begins at $(u_2,v_1)$, and the cop moves first.  We know from Case 1.~that the cop can limit the damage to at most $\max\{{\rm dmg}(C_m)|V(C_n)|,  {\rm dmg}(C_n)|V(C_m)|\}$ vertices.   Thus, if the cop begins at $(u_0,v_0)$, the robber begins at $(u_1,v_1)$ and the cop initially passes, the cop can prevent the robber from damaging more than $1+\max\{{\rm dmg}(C_m)|V(C_n)|,  {\rm dmg}(C_n)|V(C_m)|\}$ vertices.\end{proof}

\begin{theorem}\label{thm:evencyclelower}If $m,n \geq 4$ and with $m$ and $n$ both even then $${\rm dmg}(C_m \square C_n) = 1 + \max \big\{ {\rm dmg}(C_m)|V(C_n)|, {\rm dmg}(C_n)|V(C_m)| \big\}.$$\end{theorem}

\begin{proof}The upper bound follows from Lemma~\ref{thm: upperbound product of cycles}.  For the lower bound, without loss of generality, suppose the cop initially occupies $(u_0,v_0)$ on $C_m \square C_n$.  We will show that if the robber initially occupies $(u_1,v_1)$, they can damage at least $$1+\max \big\{ {\rm dmg}(C_m)|V(C_n)|, {\rm dmg}(C_n)|V(C_m)| \big\}$$ vertices, regardless of the cop's strategy.  For $i \in \{0,1,\dots,n-1\}$, the subgraph induced by the vertex set $\{(u_j,v_i)~:~ 0 \leq j \leq m-1 \}$ is called {\it copy i of $C_m$}.\medskip

{\it Phase 1:} We first assume the cop's initial move is to change their second coordinate to occupy $(u_0,v_{n-1})$ or $(u_0,v_1)$. The robber will play a shadow strategy on copy 1 of $C_m$. To do this, the robber projects the image of the cop from the second coordinate, to its corresponding position in the robber’s copy of $C_m$: the cop’s shadow.

Initially, the cop occupies $(u_0,v_0)$, so the cop's shadow on copy 1 of $C_m$ is at $(u_0,v_1)$.  The robber aims to move within copy 1 of $C_m$ and damage ${\rm dmg}(C_m)+1$ vertices while maintaining a distance of at least two from the cop's shadow after each move of the robber.  Consequently, the robber will initially move from $(u_1,v_1)$ to $(u_2,v_1)$.  Regardless of the cop's subsequent moves, the robber will, at each round, increase their first coordinate until the robber moves to $(u_{m/2},v_1)$ during  round $m/2-1$.  After the cop moves during round $m/2$, the cop's shadow is either at $(u_{m/2-1},v_1)$ or is distance at least two from the robber.  Note that the actual cop either occupies $(u_{m/2-1},v_1)$ or a vertex at least distance two from the robber.  With the robber's move during round $m/2$, vertex $(u_{m/2},v_1)$ is damaged.  As a result, at least $m/2 = {\rm dmg}(C_m)+1$ vertices are damaged in copy 1 of $C_m$.\medskip  

{\it Phase 2:} Without loss of generality, suppose the robber next moves to copy 2 of $C_m$.  Note that this results in the robber being at least distance two from the cop and at least distance one from the cop's shadow in copy 2 of $C_m$.  At the end of round $m/2$ either: \begin{enumerate}[(1)] \item the cop's shadow is at least distance two from the robber in copy 2 of $C_m$; or \item the cop occupies $(u_{m/2-1},v_1)$ and the cop's shadow is distance one from the robber in copy 2 of $C_m$. \end{enumerate}

Suppose (1) occurs.  Then the cop's shadow is at least distance two from the robber in copy 1 of $C_m$ and so after the robber moves to copy 2 of $C_m$, the cop's shadow remains at least distance two from the robber in copy 2 of $C_m$.  The robber can follow the shadow strategy and damage at least ${\rm dmg}(C_m)$ vertices in copy 2 such that at the end of each round, the robber is at least distance two from the cop's shadow.  Once this occurs, the robber moves to another copy of $C_m$ and we are essentially beginning Phase 2 again, but for a new copy of $C_m$.  Continuing in a similar fashion, in the remaining copies of $C_m$, the robber can damage at least ${\rm dmg}(C_m)$ vertices, yielding a total of at least $1+{\rm dmg}(C_m)|V(C_n)|$ damaged vertices.

Suppose (2) occurs. If the cop passes or moves to $(u_{m/2-1},v_2)$ or $(u_{m/2-1},v_0)$, then we are essentially beginning Phase 1 again, but for copy 2 of $C_m$.  In this case, the robber will damage ${\rm dmg}(C_m)+1$ vertices in copy 2 of $C_m$.  Otherwise, suppose the cop moves to $(u_{m/2},v_1)$ during round $m/2+1$.  Then immediately after the cop has moved (but before the robber has moved), the cop's shadow in copy 2 of $C_m$ coincides with the robber's position in copy 2 of $C_m$ (i.e. they have the same first coordinate).  If the cop chooses to maintain the same first coordinate as the robber forever (i.e. whenever the robber changes first coordinate, the cop then changes first coordinate to match), the robber will be able to damage $m$ vertices in copy 2 of $C_m$, $m$ vertices in copy 3 of $C_m$, and so on, to $m$ vertices in copy $n/2$ of $C_m$.  This yields $m(n/2) > {\rm dmg}(C_m)|V(C_n)|$ damaged vertices.  Otherwise, the cop may match the first coordinate as the robber for some rounds, but not others.  During the rounds where the latter occurs, we next show that the robber continues to damage at least ${\rm dmg}(C_m)$ many vertices in copies of $C_m$.

Finally, if (2) occurs and the cop moves to $(u_{m/2-2},v_1)$  during round $m/2+1$, then the cop's shadow is at least distance two from the robber in copy 2 of $C_m$ and the robber proceeds as in (1) above.\medskip

Recall that the above argument assumed the cop's first move was to change their second coordinate.  If the cop initially passes, the robber can damage $(u_1,v_1)$ and move to $(u_2,v_1)$ and following the above argument, damage at least $1+{\rm dmg}(C_m)|V(C_n)|$ vertices.  By a symmetric argument, if the cop's first move is to change their first coordinate, the robber first plays a shadow strategy on copy 1 of $C_n$ and can damage at least $1+{\rm dmg}(C_n)|V(C_m)|$ vertices.\end{proof}

\section{Graphs with small damage number}\label{sec:small}

The characterization of graphs with damage number $0$ was first observed by Cox and Sanaei~\cite{CS2019}.  

\begin{observation}\label{obs:dmg0}
For a graph $G$, ${\rm dmg}(G) = 0$ if and only if $G$ contains a universal vertex.
\end{observation}

Suppose $G$ and $H$ are graphs with universal vertices and $|V(G)|=m \leq n = |V(H)|$. To avoid trivial products, we assume both $G$ and $H$ contain at least one edge and recall that all graphs are assumed to be connected.  Then by Theorem~\ref{thm:newlowerbound} and Theorem~\ref{thm:oneuni}, ${\rm dmg}(G \square H) \in \{1,2,\dots,m\}$.  A natural question is the following.

\begin{question} For each integer $m \geq 3$, do there exist graphs $G$ and $H$ with universal vertices where $|V(G)| = m \leq n = |V(H)|$ such that ${\rm dmg}(G \square H) = k$ for each $k \in \{1,2,\dots,m\}$? 
\end{question}

By Theorem~\ref{thm:prodTrees}, a damage number of $1$ can be achieved by considering $K_{1,n} \square K_{1,m}$.  By Theorem~\ref{thm:prodK}, a damage number of $m$ can be achieved by considering the product of cliques of size $m$ and $n$ where $m \leq n$.\medskip

We next consider graphs with damage number $1$.  The following observation simplifies later results by allowing us to assume without loss of generality, that for graphs with damage number $1$, the robber remains on their starting vertex for the entirety of the game.  Suppose ${\rm dmg}(G)=1$ for graph $G$.  Then the robber either passes at every step; or at some step moves to a new vertex and is captured when the cop next moves.

\begin{observation}
\label{obs:dmg1robber}
Let $G$ be a graph. If ${\rm dmg}(G) = 1$, there is no benefit for the robber to move during the game.
\end{observation}

\begin{lemma}
\label{lemma: dmg1rad2}
For a graph $G$, if ${\rm dmg}(G) = 1$, then ${\rm rad}(G) = 2$.\end{lemma}

\begin{proof}
Suppose ${\rm dmg}(G)=1$ for graph $G$.  By Theorem~\ref{thm:newlowerbound}, ${\rm rad}(G) \in \{0,1,2\}$. If ${\rm rad}(G) \in \{0,1\}$, then $G$ is an isolated vertex or has a universal vertex and ${\rm dmg}(G)=0$.\end{proof}

Note that the converse of Lemma \ref{lemma: dmg1rad2} is not true: ${\rm rad}(C_5) = 2$ but ${\rm dmg}(C_5) \neq 1$. Using Observation \ref{obs:dmg1robber} and Lemma \ref{lemma: dmg1rad2},  we characterize graphs with damage number $1$. We note that for the class of graphs with cop number $2$, a characterization for damage number $1$ graphs was given by Carlson et al.~\cite{throttling} in the context of throttling. 

\begin{theorem}\label{thm:dmg1}
 For a graph $G$, ${\rm dmg}(G) = 1$ if and only if ${\rm rad}(G) = 2$ and a center of the graph $c \in V(G)$ is such that  for all $w \in V(G) \setminus N[c]$ there exists $s \in N[c]$ such that $s$ dominates $w$.
 \end{theorem}
 
 \begin{proof}
 For the forward direction, consider the contrapositive. Assume $G$ is a graph where either ${\rm rad}(G) \not= 2$ or for all $c \in V(G)$, there exists $w \in V(G) \setminus N[c]$ such that for each $s \in N[c]$, $N(w) \not\subseteq N[s]$. We will show that ${\rm dmg}(G) \neq 1$.
  
 If ${\rm rad}(G) = 0$, then $G$ is an isolated vertex and trivially ${\rm dmg}(G)=0$. If ${\rm rad}(G) = 1$ or ${\rm rad}(G) >2$, then ${\rm dmg}(G) \not= 1$ by the contrapositive of Lemma \ref{lemma: dmg1rad2}.  Consequently, we assume ${\rm rad}(G) = 2$ and show that no matter the strategy of the cop, the robber can damage more than one vertex in $G$. Since ${\rm rad}(G) = 2$, no matter where the cop places themselves in $G$, there always is a vertex at least distance two away. Suppose the cop initially occupies some vertex $v \in V(G)$. If there is a vertex at distance $3$ away from the cop, the robber can damage at least two vertices. Otherwise, the cop must be located on a center since the graph has radius $2$.  Given the properties of $G$, there exists $w \in V(G) \setminus N[v]$ such that 
 $N(w) \not\subseteq N(v)$ and for all $s \in N(v)$, $N(w) \not\subseteq N[s]$. We next show that if the robber starts on such a vertex $w$, they can damage at least two vertices.

 In the first round, the cop has a choice to either stay  on vertex $v$ or move to some vertex $s_i \in N(v)$ where $i \in \{1, 2,\ldots, k\}$ and $|N(v)|= k$. 

 Since $N(w) \not\subseteq N(v)$, there exists some vertex $t \in N(w)$ where $t \not\in N(v)$.  
 If the cop stays on vertex $v$ in the first round, then the robber can damage $w$ and move to vertex $t$. Since $t \not \in N(v)$, during the second round, the cop is still at least a distance of two away from the robber. No matter where the cop moves in the second round, the robber can damage vertex $t$ by remaining on it for the second round. Thus, the robber damages at least two vertices. 
  
 Suppose instead that the cop moves to vertex $s_i$ during the first round. Given the properties of $G$, we know there is some vertex $r \in N(w)$ where $r \not \in  N[s_i]$. Note that $r$ is not necessarily distinct from the vertex $t$ where $t \in N(w)$ but $t \not\in N(v)$. During the first round, the robber can damage $w$ and move to vertex $r$.  At this point, since $r \not \in N[s_i]$, the robber is still at least distance of two away from the cop.  No matter where the cop moves in the second round, the robber can damage $r$ by remaining on it for the second round.  Thus, the robber damages at least two vertices. 
 
For the reverse implication, assume that $G$ is a graph with radius $2$ and the conditions of the theorem hold. We will show that if the cop initially occupies a center vertex $c$, they have a strategy to protect $|V(G)| -1$ vertices and the robber has a strategy to damage one vertex. Since ${\rm rad}(G) = 2$, we know ${\rm dmg}(G) \geq 1$ by Theorem~\ref{thm:newlowerbound}. 

If the cop initially occupies a center $c$, then the cop is protecting $N[c]$ from damage in the first round. Therefore the robber will initially occupy some vertex $w \in V(G) \setminus N[c]$. 

By the way we have defined $G$, we know there is a vertex $s \in N[c]$ that dominates $w$.  If $s=c$ the cop can protect $|V(G)|-1$ by passing in all rounds, unless the robber moves to a neighbour of $c$.  Otherwise, the cop initially moves from $c$ to $s$ such that $s$ dominates $w$.  The cop protects $|V(G)|-1$ vertices by passing in all subsequent rounds, unless the robber moves to a neighbour of $s$.\end{proof}

Note that by the characterization of damage number $1$ graphs, for any vertex $w$ not adjacent to center $c$ either: $c$ o-dominates $w$ or there exists $s \in N(c)$ such that $s$ dominates $w$.  So to bound ${\rm dmg}(G \square H)$ for graphs $G$ and $H$ with damage number $1$, we consider two situations: at least one of the coordinates of the robber's initial position satisfies the former, or both coordinates of the robber's initial position satisfy the latter. For the first situation, we can easily bound the damage.

\begin{proposition}\label{propx} Let $G$ and $H$ be graphs for which ${\rm dmg}(G)={\rm dmg}(H)=1$.  Let $u$ and $v$ be centers in $G$ and $H$, respectively, as described in Theorem~\ref{thm:dmg1}.  If, \begin{enumerate}[(1)]

\item for all $w \notin N_G(u)$, $u$ o-dominates $w$ in $G$, or

\item for all $w \notin N_H(v)$, $v$ o-dominates $w$ in $H$; \end{enumerate}then ${\rm dmg}(G \square H) \leq \max \{ |V(G)|,|V(H)|\}$.\end{proposition} 

\begin{proof}  Suppose the cop initially occupies $(u,v)$ where $u$ and $v$ are as described in the theorem statement; and the robber initially occupies some vertex $(x,y)$. Assume $u$ o-dominates $x$ in $G$ and for now, assume $y \not\in N_H(v)$. If $v$ o-dominates $y$ in $H$, then the cop in $G \square H$ initially passes.  Otherwise, there exists $s \in N_H(v)$ such that $s$ o-dominates $y$ in $H$ and the cop in $G \square H$ moves to $(u,s)$.  Observe that the cop's first and second coordinates now o-dominate the robber's first and second coordinates, in their respective input graphs.  The robber must pass indefinitely (and damage only one vertex) or during some round, change a coordinate.  Whenever the robber changes first (second) coordinate, the cop changes first (second) coordinate to match.  Without loss of generality, suppose the robber changes second coordinate.  Then the cop changes second coordinate to match that of the robber.  If the robber now changes first coordinate, the cop will move and capture the robber.  Thus, to maximize the number of damaged vertices, the robber can only ever change second coordinate and at most $|V(H)|$ vertices are damaged.

If $y \in N_H(v)$, the cop in $G \square H$ initially moves to $(u,y)$ to match the second coordinate of the robber.  As above, if the robber changes first coordinate, the cop will capture the robber.\end{proof}

By Corollary~\ref{corcor}, if ${\rm dmg}(G)=1={\rm dmg}(H)$ and $|V(G)| \leq |V(H)|$ then $${\rm dmg}(G \square H) \leq \begin{cases} ~|V(H)| & \text{~if~$|V(H)| \geq 2|V(G)|$} \\ 2|V(G)| & \text{~if~$|V(H)|<2|V(G)|$.}\end{cases}$$ We leave as a question, whether the bounds can be improved for the latter case.

\begin{question} Let $G$ and $H$ be graphs for which ${\rm dmg}(G)={\rm dmg}(H)=1$ and suppose $|V(G)|\leq |V(H)| < 2|V(G)|$.  We further suppose that neither condition (1) nor (2) of Proposition~\ref{propx} holds.  Can the upper bound of ${\rm dmg}(G \square H) \leq  2|V(G)|$ be improved? \end{question}

Observe that any graph with damage number $2$ must have radius $2$ or radius $3$.  We next characterize graphs with damage number $2$.

\begin{lemma}
Let $G$ be a graph. If ${\rm dmg}(G) = 2$, then the radius of $G$ must be $2$ or $3$.
\end{lemma}

\begin{proof} Let $G$ be a graph for which ${\rm dmg}(G)=2$.  By Theorem~\ref{thm:newlowerbound}, ${\rm rad}(G) \in \{0,1,2,3\}$.  If ${\rm rad}(G) = 0$, then the graph has a single vertex and no vertices can be damaged. Similarly, if ${\rm rad}(G) = 1$, then there exists a universal vertex and the damage number of $G$ is $0$.  Thus, ${\rm rad}(G) \in \{2,3\}$.\end{proof}  

Note that the converse is false for both radius $2$ and radius $3$.  Cycles $C_4$ and $C_7$ are counterexamples to all radius $2$ and radius $3$ graphs (respectively) having damage number $2$.

\begin{theorem}\label{thm: dmg 2 rad 2}  Let $G$ be a graph with ${\rm rad}(G)=2$ or ${\rm rad}(G)=3$ and ${\rm dmg}(G) \neq 1$.  Then ${\rm dmg}(G)=2$ if and only if there exist vertices $z,y \in V(G)$ and $s_y \in N[z]$ such that  \begin{enumerate}[(1)]

\item ${\rm dist}_G(z,y) \in \{2,3\}$, and

\item no vertex in $N[z]$ dominates $y$, and 

\item $\forall~x \in N(y) \backslash N[s_y]$, $\exists~s_x \in N[s_y]$ and $v \in N[s_x]$ such that $$N(x)\backslash \{y\} \subseteq N[s_x]~\text{and}~N(y) \backslash \{x\} \subseteq N[v];$$ \end{enumerate} and for all $w \in V(G)\backslash \{y\}$ such that ${\rm dist}_G(z,w) \in \{2,3\}$, the above three conditions apply; or the conditions for ${\rm dmg}(G)=1$ apply.\end{theorem}

\begin{proof}
Let $G$ be a graph with radius $2$ or radius $3$ and damage number of $G$ is $2$.  We next show that the theorem conditions hold. The cop initially occupies some vertex $z \in V(G)$.

If ${\rm rad}(G)=2$, the eccentricity of $z$ is at most $4$.  If the eccentricity of $z$ equals $4$, then there is a shortest path $P=(z,v_1,v_2,v_3,v_4)$. The robber can initially occupy $v_2$ and damage $v_2,v_3,v_4$, which contradicts ${\rm dmg}(G)=2$. If ${\rm rad}(G)=3$,  the eccentricity of $z$ is at most $6$.  If the eccentricity of $z$ is $4$, $5$, or $6$, then as in the ${\rm rad}(G)=2$ case, the robber can damage at least three vertices, a contradiction. Thus, a necessary condition on $z$ is that every vertex is within distance three of $z$.  

The robber must start on a vertex at distance two or three from $z$, since otherwise the robber would be captured in the first round. Label this vertex $y$, so ${\rm dist}_G(z,y) \in \{2,3\}$.  Since ${\rm dmg}(G) = 2$,  no vertex in $N[z]$ dominates $y$. Thus, regardless of where the cop moves on the first round, say $s_y$, there must exist a safe vertex to which the robber can move. So, there is at least one vertex $x \in N(y) \backslash N[s_y]$.

The cop first moves to some vertex $s_y \in N[z]$: the closed neighbourhood allows for the possibility that the cop passes and $s_y = z$.  Then the robber moves to some vertex $x \in N(y)\backslash N[s_y]$.  Since the damage number is $2$, the robber is not captured by the next move of the cop; the cop's next move is to some vertex $s_x \in N[s_y]$ (allowing the possibility that the cop passes and $s_x = s_y$). Since the damage number is $2$, there can be no neighbour of $x$, apart from possibly $y$, that is not adjacent to $s_x$.  Thus, $N(x) \backslash \{y\} \subseteq N[s_x]$.  If $s_x \not\sim y$, the robber can move back to $y$ and so there must be a vertex $v \in N[s_x]$ to which the cop can move such that $N(y)\backslash \{x\} \subseteq N[v]$; otherwise, the robber can move to, and damage a third vertex. Thus, if ${\rm dmg}(G)=2$, conditions (1), (2), and~(3)~of the theorem must hold. The existence of the conditions implies that the robber can damage two vertices. Notably, any vertex $w$, at distance two or three from $z$ must either satisfy those conditions (to ensure the number of vertices damaged cannot exceed two), or $w$ satisfies the condition of Theorem~\ref{thm:dmg1}, which implies the damage number would be $1$ if the robber were to initially occupy such a vertex.\medskip

Considering the reverse implication, let $G$ be a graph with radius $2$ or radius $3$.  Now, suppose the conditions of the theorem are met and the cop initially occupies $z$.  Suppose the robber initially occupies a vertex $y$ satisfying conditions~(1),~(2), and~(3). Since no vertex in $N[z]$ dominates $y$, regardless of where the cop moves, the robber will be able to move to, and damage, a second vertex.  Thus, if the cop initially occupies $z$, at least two vertices will be damaged.  We next show that at most two vertices will be damaged.  The cop initially moves from $z$ to some vertex $s_y$ as described by the conditions of the theorem and the robber's only move is to some vertex $x \not\in N[s_y]$.  Recall such a vertex exists because of condition~(2).  However, by condition~(3), the cop can move to a vertex $s_x \in N[s_y]$ such that $N(x) \backslash \{y\} \subseteq N[s_x]$.  To avoid capture, the robber can either remain at $x$ (provided $s_x \not\sim x$) or move back to $y$ (provided $s_x \not\sim y$).  But, in moving back to $y$, the cop can then move to $v \in N[s_x]$ such that $N(y) \backslash \{x\} \subseteq N[v]$ by condition~(3).  Thus, the robber can only damage two vertices when the cop and robber initially occupy $z$ and $y$, respectively.\end{proof}

Let $G$ be the graph shown in Figure~\ref{fig:Rad2Dmg2}.  Observe that ${\rm rad}(G)=2$; $c$ is the only center vertex (so the only vertex with eccentricity $2$); and $z$ has eccentricity $3$.  Either by inspection or using Theorem~\ref{thm: dmg 2 rad 2}, it is easy to see that if the cop begins at $z$, the robber can damage at most two vertices.  With a little more work, one can observe that if the cop instead begins on vertex $c$, the robber can damage at least three vertices: suppose the cop and robber initially occupy $c$ and $v_5$, respectively.

\begin{itemize}\item If the cop passes or moves to $z$ or $v_3$, the robber can move to $v_8$.  Regardless of the cop's next move, the robber will be able to move to one of $v_6,v_7$ and not be adjacent to the cop.  Thus, the robber will damage at least three vertices.

\item If the cop initially moves to $v_6$, the robber can move to $v_4$.  Regardless of the cop's next move, the robber will be able to move to one of $v_3,v_7$ and not be adjacent to the cop.  Thus, the robber will damage at least three vertices.\end{itemize}

\begin{figure}[ht]
\[ 
\includegraphics[width=0.225\textwidth]{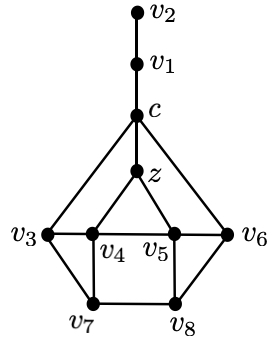}
\] 
\caption{An example of a graph with radius $2$ and damage number $2$ where the optimal starting location for the cop is not a center vertex.} 
\label{fig:Rad2Dmg2}
\end{figure}

Thus, the graph in Figure~\ref{fig:Rad2Dmg2} shows that for a graph of radius $2$, the vertex identified as $z$ in Theorem~\ref{thm: dmg 2 rad 2} is not necessarily a center vertex.  In contrast, if $G$ is a graph with radius $3$ and damage number $2$, the cop must begin on a center vertex in order to restrict the damage to two vertices.  Otherwise, if the cop begins on a vertex of eccentricity at least $4$, the robber can start distance $2$ away from the cop and damage $3$ vertices, as per the argument in the second paragraph of the proof of Theorem~\ref{thm: dmg 2 rad 2}.  More generally, if $G$ is a graph with ${\rm rad}(G)=k$ and ${\rm dmg}(G)=k-1$, the cop needs to start on a center.  But if ${\rm dmg}(G)>{\rm rad}(G)-1$ then that is not necessarily the case. While it seems that Theorem~\ref{thm: dmg 2 rad 2} could be extended to characterize graphs with damage number 3, 4, and so on, to use such characterizations, one would have to consider not just center vertices as potential starting positions for the cop, but non-center vertices too.

To conclude, we briefly comment on the product of graphs with damage number $2$.  Let $G$ and $H$ be graphs, each with damage number $2$, and for which $|V(G)|\geq |V(H)|$.  By Corollary~\ref{corcor} ${\rm dmg}(G \square H) \leq \max \{2|V(G)|,3|V(H)|\}$.  A natural question is whether there exist such graphs $G$ and $H$ (as described above) for which the damage number of the product exceeds $2|V(G)|$. Theorem~\ref{thm:evencyclelower} answers this question affirmatively: ${\rm dmg}(C_6 \square C_6) > 2|V(C_6)|$.  (In fact, by Lemma~\ref{thm: upperbound product of cycles}, ${\rm dmg}(C_6 \square C_6)=2|V(C_6)|+1$.)  We conclude with a natural question.

\begin{question} Can the upper bound of Corollary~\ref{corcor} be improved when the input graphs $G$ and $H$ both have damage number $2$? \end{question}

\section{Acknowledgements} 
M.A. Huggan acknowledges research support from AARMS and Mount Allison University. M.E. Messinger acknowledges research support from NSERC (grant application 2018-04059).

\end{document}